\documentclass[11pt]{amsart}

\usepackage[T1]{fontenc}
\usepackage{lmodern}
\usepackage{amsmath,amssymb,amsthm,mathtools}
\usepackage[margin=1.05in]{geometry}
\usepackage{microtype}
\usepackage{xcolor}
\usepackage[colorlinks=true,linkcolor=blue!55!black,citecolor=blue!55!black,urlcolor=blue!60!black]{hyperref}

\newtheorem{theorem}{Theorem}[section]
\newtheorem{proposition}[theorem]{Proposition}
\newtheorem{lemma}[theorem]{Lemma}
\newtheorem{corollary}[theorem]{Corollary}
\numberwithin{equation}{section}

\newcommand{\R}{\mathbb R}
\newcommand{\cQ}{\mathcal Q}
\newcommand{\cB}{\mathcal B}
\newcommand{\cR}{\mathcal R}
\newcommand{\pBMO}{\operatorname{pBMO}}

\newcommand{\one}{\mathbf 1}
\newcommand{\avg}[2]{\left\langle #1\right\rangle_{#2}}
\newcommand{\norm}[2]{\left\|#1\right\|_{#2}}

\title[Beyond exponential growth]
{Uniform bounds for a thermo-diffusive system beyond exponential growth}
\author{Weinan Wang}
\address{Department of Mathematics, University of Oklahoma,
Norman, Oklahoma 73019, USA}
\email{ww@ou.edu}
\subjclass[2020]{35K57, 35B45, 42B20}
\keywords{thermo-diffusive system, reaction--diffusion equation,
uniform bound, superexponential growth, parabolic BMO}
\date{}
\hypersetup{
 pdftitle={Uniform bounds for a thermo-diffusive system beyond exponential growth},
 pdfauthor={Weinan Wang},
 pdfsubject={Uniform bounds for thermo-diffusive reactions beyond exponential growth},
 pdfkeywords={thermo-diffusive system, reaction-diffusion equation,
 uniform bound, superexponential growth, parabolic BMO}
}

\begin{document}

\begin{abstract}
In this paper, we prove a global-in-time uniform bound for solutions to a thermo-diffusive system in the whole space with bounded nonnegative initial data. The result allows exponential reaction rates and a range of superexponential rates determined by the diffusivities. This extends the recent results of La, Roquejoffre, and Ryzhik \cite{LRR24} and of Leontev and Ryzhik \cite{LR26} beyond their strictly subexponential range.
\end{abstract}

\maketitle

\section{Introduction and main result}

In this paper, we are interested in the thermo-diffusive system
\begin{equation}\label{eq:system}
 \begin{split}
 u_t&=\nu\Delta u-q,\\
 v_t&=\kappa\Delta v+q,\qquad q=u g(v),
 \end{split}
 \qquad (t,x)\in(0,\infty)\times\R^d,
\end{equation}
where $d\ge1$, $\nu,\kappa>0$, and $u,v\ge0$.  In the combustion
interpretation, $u$ is the fuel concentration, $v$ is the temperature,
and $g(v)$ is the temperature-dependent reaction rate.  We assume that
the initial data are bounded:
\begin{equation}\label{eq:data}
        0\le u_0,v_0\le K,\qquad K\ge1,
\end{equation}
and impose the basic assumptions
\begin{equation}\label{eq:gassumptions}
 g\in C^1([0,\infty)),\qquad g\ge0,
 \qquad \liminf_{s\to\infty}g(s)>0.
\end{equation}
No vanishing condition at the origin is imposed on $g$.

The reaction term in \eqref{eq:system} transfers mass from the fuel
equation to the temperature equation.  The comparison principle gives
\[
                         0\le u(t,x)\le K,
\]
but the same argument does not control $v$.  This distinction is
especially important on $\R^d$: bounded initial fuel need not have finite
total mass, so fuel may be available at arbitrarily large distances, and
the whole-space Laplacian has no spectral gap.  A time-uniform temperature
bound therefore requires more than the pointwise fuel estimate or the
formal balance between the two equations. Related conservative-transfer systems with unequal diffusivities also occur
in reaction--diffusion epidemic models \cite{PW23}. 

The diffusivity ordering separates the two regimes.  Let
\[
 G_a(t,x)=\one_{\{t>0\}}(4\pi at)^{-d/2}
             \exp\!\left(-\frac{|x|^2}{4at}\right),
 \qquad S_a(t)f=G_a(t,\cdot)*f,
\]
and use convolution in space and time.  Duhamel's formula gives the two
potentials
\begin{equation}\label{eq:potentials}
 F:=G_\nu*q=S_\nu(t)u_0-u,\qquad
 H:=G_\kappa*q=v-S_\kappa(t)v_0.
\end{equation}
Consequently, we have
\begin{equation}\label{eq:Fbound}
 0\le u\le K,\qquad 0\le F\le K,
 \qquad 0\le H\le v\le K+H.
\end{equation}
If $\kappa=\nu$, then $u+v$ solves the heat equation.  More generally, if
$\kappa\le\nu$, then we see
\begin{equation}\label{eq:introkernelcomparison}
 G_\kappa(t,x)
 \le\left(\frac{\nu}{\kappa}\right)^{d/2}G_\nu(t,x),
\end{equation}
and hence we have
\[
 H\le\left(\frac{\nu}{\kappa}\right)^{d/2}F
 \le\left(\frac{\nu}{\kappa}\right)^{d/2}K.
\]
This heat-kernel comparison is the Martin--Pierre argument
\cite{MP92}.  In the opposite regime $0<\nu<\kappa$,
\begin{equation}\label{eq:introkernelratio}
 \frac{G_\kappa(t,x)}{G_\nu(t,x)}
 =\left(\frac{\nu}{\kappa}\right)^{d/2}
   \exp\!\left(
     \frac{(\kappa-\nu)|x|^2}{4\kappa\nu t}\right),
\end{equation}
which is unbounded as $|x|^2/t\to\infty$.  Thus the bounded potential
$F$ gives no direct pointwise bound for $H$ when the temperature diffuses
faster than the fuel.

Balance-law and duality methods for reaction--diffusion systems on
bounded domains were developed in \cite{HMP87}.  The whole-space problem
with bounded, nondecaying initial data has a different large-time
character.  For polynomial reactions, Collet and Xin obtained the
asymptotic estimate
\[
                       \|v(t)\|_{L^\infty}=O(\log\log t)
                       \qquad (t\to\infty)
\]
in the difficult diffusivity regime \cite{CX96}.  For reaction rates of
exponential size, Herrero, Lacey, and Vel\'azquez proved global existence
for bounded initial data for a broad class containing
\[
                              q=u e^v,
\]
which arises in the Frank--Kamenetskii approximation of Arrhenius-type
reactions \cite{HLV98}.  Their theorem establishes global existence but
does not state a bound uniform as $t\to\infty$.

La, Roquejoffre, and Ryzhik  \cite{LRR24} subsequently established a time-uniform bound
under the subexponential condition
\[
                 g(s)\le c_2\exp(Zs^\rho),\qquad 0\le\rho<1,
\]
together with $g(s)>0$ for $s>0$, the nondegeneracy condition $g'(0)>0$,
and a uniform positive lower bound on $u_0+v_0$.  Leontev and
Ryzhik removed the strict-positivity, nondegeneracy, and initial-density
assumptions and included ignition reactions, while retaining $0<\rho<1$
and the condition $g(0)=0$ \cite{LR26}.  The hypotheses
\eqref{eq:gassumptions} do not impose the latter condition.  Both results
therefore stop at strictly subexponential growth.

The present paper addresses the exponential endpoint.  When
$0<\nu<\kappa$, set
\begin{equation}\label{eq:limitingpower}
                         \rho_*:=\frac{\kappa}{\kappa-\nu}
                         =1+\frac{\nu}{\kappa-\nu}>1.
\end{equation}

\begin{theorem}[Uniform bound beyond exponential growth]\label{thm:main}
Let $d\ge1$, let $0<\nu<\kappa$, and assume
\eqref{eq:data}--\eqref{eq:gassumptions}.  Suppose that there are
$c_2,Z>0$ and $0<\rho<\rho_*$ such that
\begin{equation}\label{eq:growthassumption}
             g(s)\le c_2\exp(Zs^\rho),\qquad s\ge0.
\end{equation}
Then \eqref{eq:system} has a unique nonnegative global mild solution in the
class of mild solutions bounded on compact time intervals.  This solution is
classical for positive time, and there is a finite constant $C>0$ such that
\[
                 \sup_{t\ge0,\,x\in\R^d}v(t,x)\le C.
\]
If $\vartheta>0$ is chosen so that
\[
 h_\vartheta:=\inf_{s>\vartheta}g(s)>0,
\]
then $C$ may be chosen to depend only on
$d,\nu,\kappa,K,c_2,Z,\rho,\vartheta$, and $h_\vartheta$.
\end{theorem}

In particular, Theorem~\ref{thm:main} includes the
Frank--Kamenetskii exponential law $g(s)=A e^{Zs}$ for every $A,Z>0$,
which lies outside the subexponential hypotheses of \cite{LRR24,LR26}.  The
proof overcomes the fixed-exponent obstruction in the John--Nirenberg lifting
step.  More
generally, the interval $1<\rho<\rho_*$ permits
superexponential reactions.  For $q=Aue^{Zv}$, the change of variables
$\widehat t=At$, $U=Zu$, and $V=Zv$ reduces the system to the normalized
Frank--Kamenetskii system treated in \cite{HLV98}.  For the continuous,
bounded, nonnegative initial data considered there, \cite{HLV98} proves
global existence; Theorem~\ref{thm:main} gives the corresponding
time-uniform temperature bound and extends global solvability to the full
range $0<\rho<\rho_*$ under \eqref{eq:gassumptions}.

The exponent $\rho_*$ is determined by the two Gaussian rates in the cone
decomposition.  The near-cone comparison incurs the factor
\[
 \exp\!\left(\frac{(1+\delta)(\kappa-\nu)m^2}{4\kappa\nu}\right),
\]
whereas the far-cone parabolic-BMO estimate gains
$\exp(-m^2/(4\kappa))$.  Balancing these rates gives
\[
 \beta_\delta=1+\frac{\nu}{(1+\delta)(\kappa-\nu)},
 \qquad \beta_\delta\uparrow\rho_*
 \quad\text{as }\delta\downarrow0.
\]
The dependence of $\rho_*$ on the diffusivity ratio is consistent with the
limiting regimes:
\[
 \rho_*\downarrow1\quad\text{as }\nu/\kappa\downarrow0,
 \qquad
 \rho_*\uparrow\infty\quad\text{as }\nu/\kappa\uparrow1.
\]
In the equal-diffusivity limit, the latter behavior agrees with the direct
heat-kernel comparison, which imposes no upper growth restriction on $g$.

The restriction $\rho<1$ in the preceding results enters at the local
lifting step.  The singular-integral relation between the two potentials
gives the growth-independent estimate
\[
                         \|H\|_{\pBMO}\le CK.
\]
By the parabolic John--Nirenberg inequality \cite{JN61,Aim88}, this controls
\[
 \avg{\exp\!\left(
   \lambda\left|H-\avg{H}{Q}\right|\right)}{Q}
\]
only while $\lambda\|H\|_{\pBMO}$ remains below a dimensional threshold.
If $0<\rho<1$, the inequality
\[
                       s^\rho\le\epsilon s+C_\epsilon
\]
reduces the required coefficient to an arbitrarily small one.  This
reduction is unavailable at and beyond $\rho=1$.

For $0<\nu<\kappa$, we split the temperature potential according to the
parabolic cone in which the kernel comparison
\eqref{eq:introkernelratio} is effective.  Given $0<\delta\le1$, the
cutoff changes across the thin annulus
\[
                m^2t<|x|^2<(1+\delta)m^2t.
\]
For every $m\ge2$ this gives
\begin{equation}\label{eq:introdecomposition}
 H=\cB_{m,\delta}+\cR_{m,\delta},
\end{equation}
where
\begin{equation}\label{eq:introdecompositionbounds}
 0\le\cB_{m,\delta}\le D_{m,\delta}K,\qquad
 0\le\cR_{m,\delta}\le H,
\end{equation}
with
\[
 D_{m,\delta}=\left(\frac{\nu}{\kappa}\right)^{d/2}
 \exp\!\left(\frac{(1+\delta)(\kappa-\nu)m^2}
                         {4\kappa\nu}\right)
\]
and
\begin{equation}\label{eq:introsmallbmo}
 \|\cR_{m,\delta}\|_{\pBMO}
 \le C\delta^{-M}(1+m)^P
              e^{-m^2/(4\kappa)}K.
\end{equation}
The near-cone potential is bounded by comparing the two heat kernels with
the finite factor $D_{m,\delta}$.  For the far-cone potential, the
kernel obtained by applying $\partial_t-\nu\Delta$ has parabolic
Calder\'on--Zygmund cancellation.  Tracking the Gaussian tail gives the
small seminorm in \eqref{eq:introsmallbmo}.

A fixed aperture already suffices at the exponential endpoint.  Indeed, for
any $0<\Lambda<\infty$, fix $\delta=1$ and choose $m$ so large that
\[
 \Lambda C(1+m)^P e^{-m^2/(4\kappa)}K
\]
lies below the John--Nirenberg threshold.  If $\avg{v}{Q}\le A$, then
$v\le K+D_{m,1}K+\cR_{m,1}$ and
$\avg{\cR_{m,1}}{Q}\le A$, whence
\begin{equation}\label{eq:introfixedaperture}
 \avg{e^{\Lambda v}}{Q}
 \le C_{\mathrm{JN}}
       \exp\!\left(\Lambda\{K+D_{m,1}K+A\}\right).
\end{equation}
Since $s^\rho\le1+s$ for $0<\rho<1$, a fixed aperture gives the range
$0<\rho\le1$.  To obtain superexponential moments, the cone aperture is
chosen as a function of the temperature level.
If $Q$ satisfies $\avg{v}{Q}\le A$, the John--Nirenberg inequality and
\eqref{eq:introdecompositionbounds}--\eqref{eq:introsmallbmo} yield
\begin{equation}\label{eq:introtail}
 \frac{|\{v>\lambda\}\cap Q|}{|Q|}
 \le C_\delta\exp\!\left[-c_{\delta,K}
 \frac{(\lambda/K)^{\beta_\delta}}
      {\{\log(e+\lambda/K)\}^{P/2}}\right],
 \qquad
 \beta_\delta=1+\frac{\nu}{(1+\delta)(\kappa-\nu)},
\end{equation}
for all sufficiently large $\lambda$, with constants independent of the
radius and location of $Q$.  Given $\rho<\rho_*$, one fixes $\delta>0$ so that
$\rho<\beta_\delta$.  The estimate \eqref{eq:introtail} then gives
\begin{equation}\label{eq:introarbitrarymoment}
 \avg{e^{\Lambda v^\rho}}{Q}
 \le E(d,\nu,\kappa,K,A,\rho,\Lambda)
 \qquad\text{for every }\Lambda<\infty.
\end{equation}
Taking $\Lambda=pZ$ with $p>(d+2)/2$ controls $v$ and $q$ in $L^p(Q)$.
An interior parabolic estimate lifts the average bound to a pointwise
bound for $v$.

It remains to obtain cylinder averages uniformly in time.  On a cold
cylinder, the local lifting estimate applies directly.  On a hot
cylinder, the fuel-loss estimate of \cite{LR26} contracts the fuel over a
fixed time step.  The smoothed Martin--Pierre estimate bounds the
accumulated heat source during that step by a heat average of the earlier
fuel concentration.  The discrete hot--cold expansion introduced in
\cite[Section~3.2]{LR26}
telescopes the cold contributions and contracts successive hot
contributions.  It follows that $S_\kappa(\tau)v(t)$ is uniformly
bounded, and hence that the unit-cylinder averages of $v$ are uniformly
bounded.  Applying the local moment and parabolic estimates completes the
proof.

Section~2 proves the cone decomposition and its quantitative
parabolic-BMO estimate.  Section~3 derives the distribution estimate
\eqref{eq:introtail}, the local moments below $\rho_*$, and the associated
lifting estimate.  Section~4 combines that estimate with the hot--cold
iteration of \cite{LR26} to prove Theorem~\ref{thm:main}.

We close the introduction by recording the standard local theory and
continuation criterion used in the proof of Theorem~\ref{thm:main}.

\begin{lemma}[Local solvability and continuation]\label{lem:local}
Assume that $u_0,v_0\in L^\infty(\R^d)$ are nonnegative and that
$g\in C^1([0,\infty))$ is nonnegative.  Then there are
$T_{\max}\in(0,\infty]$ and a unique maximal nonnegative mild solution
$(u,v)$ on $[0,T_{\max})$; uniqueness is understood in the class of mild
solutions bounded on compact time intervals.  For every $T<T_{\max}$,
\[
 (u,v)\in L^\infty((0,T)\times\R^d)^2
 \cap C((0,T];\operatorname{BUC}(\R^d))^2,
\]
the Duhamel identities hold, and $(u,v)$ is classical for positive time.
Moreover,
\[
 0\le u(t,x)\le \|u_0\|_{L^\infty},
 \qquad 0<t<T_{\max},
\]
and the continuation alternative is
\begin{equation}\label{eq:continuation}
 T_{\max}<\infty
 \quad\Longrightarrow\quad
 \limsup_{t\uparrow T_{\max}}\|v(t)\|_{L^\infty}=\infty.
\end{equation}
\end{lemma}

\begin{proof}
Extend $g$ to a nonnegative locally Lipschitz function on $\R$.  On
$L^\infty((0,T)\times\R^d)^2$, consider the Duhamel map
\[
 \begin{split}
 \Phi_1(u,v)(t)
 &=S_\nu(t)u_0-\int_0^tS_\nu(t-s)\bigl[u(s)g(v(s))\bigr]\,ds,\\
 \Phi_2(u,v)(t)
 &=S_\kappa(t)v_0+\int_0^tS_\kappa(t-s)
       \bigl[u(s)g(v(s))\bigr]\,ds.
 \end{split}
\]
The heat semigroups are contractions on $L^\infty$, and
$(a,b)\mapsto ag(b)$ is locally Lipschitz.  On every fixed bounded ball,
$\Phi$ therefore maps the ball into itself and is a contraction when
$T>0$ is sufficiently small.  This gives a unique maximal bounded mild
solution.  On a compact subinterval of its lifespan, set
$a(t,x)=g(v(t,x))$.  The chosen extension of $g$ makes $a$ bounded,
measurable, and nonnegative.  The first component is the unique mild
solution of
\[
 u_t-\nu\Delta u+a(t,x)u=0.
\]
Positivity of the corresponding linear evolution, for example from the
Feynman--Kac formula, gives
\[
 0\le u(t)\le S_\nu(t)u_0\le\|u_0\|_{L^\infty}.
\]
The second Duhamel identity then gives
\[
 v(t)=S_\kappa(t)v_0+
 \int_0^tS_\kappa(t-s)[u(s)g(v(s))]\,ds\ge0.
\]

On such a subinterval the reaction term is bounded.  Interior
$W^{2,1}_p$ estimates first give H\"older continuity for every positive
time; since $g\in C^1$, the reaction term is then locally H\"older, and
Schauder regularity makes the solution classical.  The heat-semigroup
representation also gives continuity into $\operatorname{BUC}(\R^d)$
away from $t=0$.  Finally, the contraction
argument can be restarted at any positive time, with a lifespan depending
only on an $L^\infty$ bound for $(u,v)$ and on the Lipschitz constant of $g$
on the corresponding compact range.  Hence a finite maximal time is possible
only if $\|u(t)\|_{L^\infty}+\|v(t)\|_{L^\infty}$ is unbounded as
$t\uparrow T_{\max}$.  The bound for $u$ reduces this to
\eqref{eq:continuation}.
\end{proof}

\begin{proposition}[The easy diffusivity regime]\label{prop:easy}
Assume $0<\kappa\le\nu$, the data condition \eqref{eq:data}, and
$g\in C^1([0,\infty))$, $g\ge0$.  Then the maximal solution from
Lemma~\ref{lem:local} is global and
\[
 \sup_{t\ge0,\,x\in\R^d}v(t,x)
 \le K+\left(\frac{\nu}{\kappa}\right)^{d/2}K.
\]
\end{proposition}

\begin{proof}
On the maximal lifespan, \eqref{eq:introkernelcomparison} and the Duhamel
identities give
\[
 H=G_\kappa*q
 \le\left(\frac{\nu}{\kappa}\right)^{d/2}G_\nu*q
 \le\left(\frac{\nu}{\kappa}\right)^{d/2}K.
\]
Thus $v\le K+(\nu/\kappa)^{d/2}K$.  The continuation alternative in
Lemma~\ref{lem:local} gives global existence.
\end{proof}

\section{The far-cone small-BMO decomposition}

Throughout this section assume $0<\nu<\kappa$.  Let
$[0,T_{\max})$ be the maximal lifespan from Lemma~\ref{lem:local}, fix
$T<T_{\max}$, and set
\[
 q^{(T)}(t,x)=\one_{(0,T)}(t)u(t,x)g(v(t,x)),
 \qquad F^{(T)}=G_\nu*q^{(T)},
 \qquad H^{(T)}=G_\kappa*q^{(T)}.
\]
Also define
\[
 u^{(T)}(t)=S_\nu(t)u_0-F^{(T)}(t),
 \qquad
 v^{(T)}(t)=S_\kappa(t)v_0+H^{(T)}(t).
\]
For $0\le t\le T$, the Duhamel identities give
\[
 u^{(T)}(t)=u(t),\qquad v^{(T)}(t)=v(t).
\]
For $t\ge T$, the semigroup property gives
\[
 \begin{split}
 F^{(T)}(t)
 &=S_\nu(t-T)F^{(T)}(T)
   =S_\nu(t)u_0-S_\nu(t-T)u(T),\\
 u^{(T)}(t)&=S_\nu(t-T)u(T),\\
 v^{(T)}(t)&=S_\kappa(t-T)v(T).
 \end{split}
\]
It follows that
\begin{equation}\label{eq:finitehorizonF}
 0\le F^{(T)}\le S_\nu(t)u_0\le K
 \qquad\text{on }(0,\infty)\times\R^d.
\end{equation}
Thus the whole-spacetime estimates below apply to the truncated source and
its potentials with constants independent of $T$.  On cylinders contained
in $(0,T)\times\R^d$, these quantities agree with the original solution and
$q^{(T)}=u^{(T)}g(v^{(T)})$.  We suppress the superscript $(T)$.

Write $\cQ=d+2$ for the parabolic homogeneous dimension.  For
$z_0=(t_0,x_0)$ and $r>0$, define
\[
 Q^-_r(z_0)=(t_0-r^2,t_0)\times B_r(x_0),
 \qquad
 Q_r(z_0)=(t_0-r^2,t_0+r^2)\times B_r(x_0),
 \qquad
 \avg{f}{Q}=\frac1{|Q|}\int_Qf.
\]
For a locally integrable function $f$, let
\[
 \|f\|_{\pBMO}
 :=\sup_Q\frac1{|Q|}\int_Q|f-\avg{f}{Q}|,
\]
where the supremum is taken over the symmetric cylinders $Q_r(z_0)$.
On $\R^{d+1}$, the seminorm defined using backward cylinders is equivalent,
with constants depending only on $d$.  We use this equivalence when
restricting the whole-spacetime estimates to the backward cylinders in
Section~3.  Each bounded potential used below is extended by zero to
negative times.

Fix once and for all $\psi\in C^\infty(\R)$ such that
\[
 0\le\psi\le1,\qquad \psi(s)=0\ (s\le0),\qquad
 \psi(s)=1\ (s\ge1).
\]
For $0<\delta\le1$, put
\[
 \chi_\delta(s)=\psi\!\left(\frac{s-1}{\delta}\right).
\]
Thus $\chi_\delta=0$ on $(-\infty,1]$, $\chi_\delta=1$ on
$[1+\delta,\infty)$, and
$\|\chi_\delta^{(j)}\|_\infty\le C_j\delta^{-j}$.  For $m\ge2$ and
$t>0$, define
\begin{equation}\label{eq:cutoff}
 \chi_{m,\delta}(t,x)
   =\chi_\delta\!\left(\frac{|x|^2}{m^2t}\right),
 \qquad
 L_{m,\delta}(t,x)=G_\kappa(t,x)\chi_{m,\delta}(t,x),
\end{equation}
and set $L_{m,\delta}=0$ for $t\le0$.  Split
\begin{equation}\label{eq:split}
 H=\cB_{m,\delta}+\cR_{m,\delta},\qquad
 \cB_{m,\delta}=(G_\kappa-L_{m,\delta})*q,
 \qquad \cR_{m,\delta}=L_{m,\delta}*q.
\end{equation}

\begin{proposition}[Quantitative cone decomposition]
\label{prop:decomp}
There are constants $C>0$ and integers $M,P\ge1$, depending only on
$d,\nu$, and $\kappa$, such that, for every $0<\delta\le1$ and
$m\ge2$,
\begin{equation}\label{eq:decompbounds}
 \begin{aligned}
 0&\le \cB_{m,\delta}\le D_{m,\delta}K,
 &D_{m,\delta}
   &=\left(\frac{\nu}{\kappa}\right)^{d/2}
     \exp\!\left(\frac{(1+\delta)(\kappa-\nu)m^2}
                            {4\kappa\nu}\right),\\
 0&\le \cR_{m,\delta}\le H,
 &\norm{\cR_{m,\delta}}{\pBMO}
   &\le C\delta^{-M}(1+m)^P
                    e^{-m^2/(4\kappa)}K.
 \end{aligned}
\end{equation}
\end{proposition}

\begin{proof}
Positivity gives the decomposition and $0\le\cR_{m,\delta}\le H$.  On
the support of $1-\chi_{m,\delta}$ one has
$|x|^2/t\le(1+\delta)m^2$.  Since $\kappa>\nu$,
\[
 \frac{G_\kappa(t,x)}{G_\nu(t,x)}
 =\left(\frac{\nu}{\kappa}\right)^{d/2}
   \exp\!\left(\frac{(\kappa-\nu)|x|^2}
                       {4\kappa\nu t}\right)
 \le D_{m,\delta}.
\]
Therefore
$\cB_{m,\delta}\le D_{m,\delta}(G_\nu*q)\le D_{m,\delta}K$.

We next estimate the far component.  Set
\[
 \ell_{m,\delta}(z)
  =(4\pi\kappa)^{-d/2}e^{-|z|^2/(4\kappa)}
       \chi_\delta(|z|^2/m^2),
 \qquad
 a_{m,\delta}=\int_{\R^d}\ell_{m,\delta}(z)\,dz.
\]
Then
\[
 L_{m,\delta}(t,x)
  =\one_{\{t>0\}}t^{-d/2}\ell_{m,\delta}(x/\sqrt t),
 \qquad
 L_{m,\delta}(t,\cdot)\longrightarrow a_{m,\delta}\delta_0
 \quad(t\downarrow0).
\]
For $t>0$, direct differentiation gives
\begin{equation}\label{eq:km}
 (\partial_t-\nu\Delta)L_{m,\delta}(t,x)
   =t^{-\cQ/2}k_{m,\delta}(x/\sqrt t),
\end{equation}
where
\[
 k_{m,\delta}(z)
 =-\frac d2\ell_{m,\delta}(z)
  -\frac12z\cdot\nabla\ell_{m,\delta}(z)
  -\nu\Delta\ell_{m,\delta}(z).
\]
Integration by parts gives
\begin{equation}\label{eq:cancellation}
 \int_{\R^d}k_{m,\delta}(z)\,dz=0.
\end{equation}
Moreover, $k_{m,\delta}$ is even, and hence
\begin{equation}\label{eq:firstmoments}
 \int_{\R^d}z_jk_{m,\delta}(z)\,dz=0
 \qquad(1\le j\le d).
\end{equation}

The expression on the right of \eqref{eq:km} is not locally integrable at
$(0,0)$.  For $\varphi\in C_c^\infty(\R^{d+1})$, define its canonical
principal-value extension by
\begin{equation}\label{eq:pvdefinition}
 \left\langle\widetilde K_{m,\delta},\varphi\right\rangle
 :=\lim_{\epsilon\downarrow0}
   \int_\epsilon^\infty\!\int_{\R^d}
      t^{-\cQ/2}k_{m,\delta}(x/\sqrt t)
      \varphi(t,x)\,dx\,dt.
\end{equation}
To see that the limit exists, substitute $x=\sqrt t\,z$.  By
\eqref{eq:cancellation}--\eqref{eq:firstmoments}, one may replace
$\varphi(t,\sqrt t\,z)$ by
\[
 \varphi(t,\sqrt t\,z)-\varphi(t,0)
       -\sqrt t\,z\cdot\nabla_x\varphi(t,0),
\]
which is $O(t(1+|z|^2))$ near $t=0$.  An integration by parts in $t$ and
$x$ then gives, in $\mathcal D'(\R^{d+1})$,
\begin{equation}\label{eq:Kmprofile}
 K_{m,\delta}:=(\partial_t-\nu\Delta)L_{m,\delta}
 =a_{m,\delta}\delta_{(0,0)}+\widetilde K_{m,\delta}.
\end{equation}

Fix an integer $J\ge d+10$.  The product and chain rules, together with
the derivative bounds for $\chi_\delta$, give integers $M,P\ge1$ such
that, for $|\alpha|\le J$,
\[
 |\partial^\alpha k_{m,\delta}(z)|
 \le C\delta^{-M}(1+|z|)^P
       e^{-|z|^2/(4\kappa)}\one_{\{|z|\ge m\}}.
\]
For every fixed $L\ge0$, the elementary Gaussian-tail bound
\[
 \int_{|z|\ge m}(1+|z|)^L e^{-|z|^2/(4\kappa)}\,dz
 \le C_L(1+m)^{L+d}e^{-m^2/(4\kappa)},
 \qquad m\ge2,
\]
then yields, after enlarging $P$ if necessary,
\begin{equation}\label{eq:schwartzsmall}
 \begin{split}
 a_{m,\delta}
 &+\max_{|\alpha|\le J}\left\{
  \|(1+|z|)^J\partial^\alpha k_{m,\delta}\|_{L^1_z}
  +\|(1+|z|)^J\partial^\alpha k_{m,\delta}\|_{L^\infty_z}
 \right\}\\
 &\hspace{35mm}\le
 C\delta^{-M}(1+m)^Pe^{-m^2/(4\kappa)}
 =:C\varepsilon_{m,\delta}.
 \end{split}
\end{equation}
Taylor's formula at the origin, using
\eqref{eq:cancellation}--\eqref{eq:firstmoments}, and integration by parts
at high frequency yield
\[
 \widehat{k_{m,\delta}}(\zeta)
 =\frac12\zeta\cdot\nabla\widehat{\ell_{m,\delta}}(\zeta)
  +\nu|\zeta|^2\widehat{\ell_{m,\delta}}(\zeta),
\]
and
\begin{equation}\label{eq:fourierprofile}
 |\widehat{k_{m,\delta}}(\zeta)|
 \le C\varepsilon_{m,\delta}
       \min\{|\zeta|^2,(1+|\zeta|)^{-J}\}.
\end{equation}

The spacetime Fourier multiplier of the principal-value part is
\[
 M_{m,\delta}(\tau,\xi)
 =\int_0^\infty e^{-i\tau s}s^{-1}
       \widehat{k_{m,\delta}}(\sqrt s\,\xi)\,ds.
\]
For $\xi\ne0$, the change of variables $a=s|\xi|^2$ and
\eqref{eq:fourierprofile} give
\[
 |M_{m,\delta}(\tau,\xi)|
 \le C\varepsilon_{m,\delta}\left(
       \int_0^1da+\int_1^\infty a^{-1-J/2}\,da\right)
 \le C\varepsilon_{m,\delta};
\]
for $\xi=0$ the multiplier vanishes.  The truncated multipliers
\[
 M_{m,\delta}^{\epsilon,R}(\tau,\xi)
 =\int_\epsilon^R e^{-i\tau s}s^{-1}
       \widehat{k_{m,\delta}}(\sqrt s\,\xi)\,ds
\]
obey the same bound uniformly in $0<\epsilon<R$, and converge pointwise as
$\epsilon\downarrow0$ and $R\uparrow\infty$.  Plancherel and dominated
convergence therefore define an $L^2$ operator $\mathcal T_{m,\delta}$
with norm at most $C\varepsilon_{m,\delta}$.

Away from the origin, extend the kernel in \eqref{eq:pvdefinition} by zero
to $t<0$.  It is smooth across $\{t=0,x\ne0\}$ because the Gaussian tail
vanishes there to infinite order.  Writing
$\varrho(t,x)=|x|+\sqrt{|t|}$, differentiation and
\eqref{eq:schwartzsmall} give the mixed-homogeneity estimates
\begin{equation}\label{eq:CZ}
 |\nabla_x^\alpha\partial_t^\beta
       \widetilde K_{m,\delta}(t,x)|
 \le C_{\alpha,\beta}\varepsilon_{m,\delta}
       \varrho(t,x)^{-\cQ-|\alpha|-2\beta}
\end{equation}
for $|\alpha|\le1$ and $0\le\beta\le1$.  These are the standard
parabolic Calder\'on--Zygmund bounds; see \cite{FR66} and, for the
corresponding heat-potential argument, \cite[Lemma~2.2]{LRR24}.

We prove the $L^\infty\to\pBMO$ estimate by the standard local--far
decomposition.  First let $f\in L_c^\infty$.  Given a
symmetric parabolic cylinder $Q$ of radius $r$ and center $z_Q$, write
$f=f_1+f_2$, where $f_1=f\one_{4Q}$.  The multiplier estimate and
Cauchy--Schwarz give
\[
 \frac1{|Q|}\int_Q|\mathcal T_{m,\delta}f_1|
 \le C\varepsilon_{m,\delta}\|f\|_\infty.
\]
For $z\in Q$ and $w\notin4Q$, \eqref{eq:CZ} and the parabolic mean-value
estimate imply
\[
 |\widetilde K_{m,\delta}(z-w)
       -\widetilde K_{m,\delta}(z_Q-w)|
 \le C\varepsilon_{m,\delta}
       \frac{\varrho(z-z_Q)}{\varrho(z_Q-w)^{\cQ+1}}.
\]
Taking $c_Q=\mathcal T_{m,\delta}f_2(z_Q)$ and summing over parabolic
annuli gives
\[
 \frac1{|Q|}\int_Q|\mathcal T_{m,\delta}f-c_Q|
 \le C\varepsilon_{m,\delta}\|f\|_\infty.
\]
Oscillation about the average is at most twice the oscillation about any
constant.  Approximation by common expanding truncations consequently
extends $\mathcal T_{m,\delta}$ from $L^\infty_c$ to $L^\infty$, modulo
constants, and gives
\begin{equation}\label{eq:LinftoBMO}
 \|\mathcal T_{m,\delta}f\|_{\pBMO(\R^{d+1})}
 \le C\varepsilon_{m,\delta}\|f\|_{L^\infty(\R^{d+1})}.
\end{equation}

To apply this estimate to the bounded, nondecaying potential $F$, choose
increasing spacetime truncations
\[
 q_R=q\,\one_{\{0<s<R,\ |y|<R\}},\qquad
 F_R=G_\nu*q_R,\qquad
 \cR_{m,\delta,R}=L_{m,\delta}*q_R.
\]
Then $0\le F_R\le F\le K$.  On every finite interval $(0,T)$, the heat
semigroup contraction and Young's inequality in time give
\[
 \|F_R\|_{L^2((0,T)\times\R^d)}
 \le T\|q_R\|_{L^2((0,T)\times\R^d)}.
\]
Since the kernels are causal, the restriction of
$\widetilde K_{m,\delta}*F_R$ to $(0,T)$ is the $L^2$ realization of
$\mathcal T_{m,\delta}(F_R\one_{(0,T)})$ there.  Indeed, at a time
$0<t<T$ every causal truncation samples only $F_R$ on $(0,t)$; extending
$F_R\one_{(0,T)}$ by zero therefore leaves the truncated convolutions
unchanged, and the identity follows by passage to the $L^2$ limit.  The
spatial Fourier transform and the time Laplace transform give
\[
 \mathcal{FL}[K_{m,\delta}](\lambda,\xi)
 = (\lambda+\nu|\xi|^2)
       \mathcal{FL}[L_{m,\delta}](\lambda,\xi),
 \qquad
 \mathcal{FL}[F_R](\lambda,\xi)
 =\frac{\mathcal{FL}[q_R](\lambda,\xi)}
        {\lambda+\nu|\xi|^2}
\]
for $\operatorname{Re}\lambda>0$.  Consequently,
\begin{equation}\label{eq:ibp}
 \cR_{m,\delta,R}
 =L_{m,\delta}*q_R
 =a_{m,\delta}F_R+\mathcal T_{m,\delta}F_R
\end{equation}
in $\mathcal D'(\R^{d+1})$ and almost everywhere.  Extending $F_R$ by
zero to $t<0$, one has
$(\partial_t-\nu\Delta)F_R=q_R$ in distributions with no initial trace
measure; the term $a_{m,\delta}F_R$ in \eqref{eq:ibp} comes instead from
the trace of $L_{m,\delta}$ in \eqref{eq:Kmprofile}.

Equations \eqref{eq:schwartzsmall}, \eqref{eq:LinftoBMO}, and
\eqref{eq:ibp} imply
\[
 \|\cR_{m,\delta,R}\|_{\pBMO(\R^{d+1})}
 \le C\varepsilon_{m,\delta}K
\]
uniformly in $R$.  Finally,
$\cR_{m,\delta,R}\uparrow\cR_{m,\delta}$ pointwise.  Since
$0\le\cR_{m,\delta}\le H\le v$, monotone convergence gives local $L^1$
convergence and hence convergence of the mean oscillation on every fixed
cylinder.  This proves \eqref{eq:decompbounds}.
\end{proof}

\begin{corollary}[Growth-independent BMO bound]\label{cor:globalBMO}
There is a constant $C_{\mathrm B}=C(d,\nu,\kappa,K)$ such that
\begin{equation}\label{eq:globalBMO}
 \|v\|_{\pBMO((0,\infty)\times\R^d)}\le C_{\mathrm B}.
\end{equation}
\end{corollary}

\begin{proof}
A growth-independent BMO estimate of this form was established in
\cite[Lemma~2.2]{LRR24}.  The present corollary records that
Proposition~\ref{prop:decomp} recovers it directly.  Indeed, fix $m=2$
and $\delta=1$ in Proposition~\ref{prop:decomp}.  The bounded
terms $S_\kappa(t)v_0$ and $\cB_{2,1}$ have BMO seminorms at most twice
their $L^\infty$ norms, while \eqref{eq:decompbounds} controls
$\cR_{2,1}$.  The identity
$v=S_\kappa(t)v_0+\cB_{2,1}+\cR_{2,1}$ proves the claim.
\end{proof}

\section{Local superexponential moments and pointwise lifting}

Continue to assume $0<\nu<\kappa$, and retain the finite-horizon notation
of Section~2.

\begin{proposition}[Local moments in the extended range]
\label{prop:expmoment}
Let $Q=Q^-_r(z_0)\subset(0,\infty)\times\R^d$ and suppose
\begin{equation}\label{eq:meanassume}
 \avg{v}{Q}\le A.
\end{equation}
For every $0<\rho<\rho_*$ and every $\Lambda>0$, there is a finite
constant $E=E(d,\nu,\kappa,K,A,\rho,\Lambda)$ such that
\begin{equation}\label{eq:expmoment}
 \avg{\exp(\Lambda v^\rho)}{Q}\le E.
\end{equation}
The constant is independent of the radius and location of $Q$.
\end{proposition}

\begin{proof}
Choose $0<\delta\le1$ so small that
\begin{equation}\label{eq:deltachoice}
 \rho<\beta_\delta
 :=1+\frac{\nu}{(1+\delta)(\kappa-\nu)}.
\end{equation}
Set
\[
 c_0=\left(\frac{\nu}{\kappa}\right)^{d/2},
 \qquad
 a_\delta=\frac{(1+\delta)(\kappa-\nu)}{4\kappa\nu},
 \qquad b=\frac1{4\kappa}.
\]
For $\lambda$ sufficiently large, choose the real number $m=m(\lambda)$
by
\begin{equation}\label{eq:mchoice}
 m^2=a_\delta^{-1}
       \log\!\left(\frac{\lambda}{4c_0K}\right).
\end{equation}
Increasing the lower threshold for $\lambda$ ensures that $m\ge2$,
$\lambda\ge4K$, and $\lambda\ge4A$.  By construction,
$D_{m,\delta}K=\lambda/4$.  Proposition~\ref{prop:decomp} and
\eqref{eq:Fbound} therefore imply
\[
 \{v>\lambda\}\cap Q
 \subset
 \left\{\cR_{m,\delta}-\avg{\cR_{m,\delta}}{Q}
                         >\frac\lambda4\right\}\cap Q,
\]
because $\avg{\cR_{m,\delta}}{Q}\le\avg{v}{Q}\le A$.

The parabolic John--Nirenberg inequality \cite{JN61,Aim88} and
\eqref{eq:decompbounds} now give
\begin{equation}\label{eq:distributionraw}
 \frac{|\{v>\lambda\}\cap Q|}{|Q|}
 \le C\exp\!\left[-c\delta^M
       \frac{(\lambda/K)e^{bm^2}}{(1+m)^P}\right].
\end{equation}
Since $b/a_\delta=\nu/((1+\delta)(\kappa-\nu))$, substitution of
\eqref{eq:mchoice} into \eqref{eq:distributionraw} yields
\begin{equation}\label{eq:distribution}
 \frac{|\{v>\lambda\}\cap Q|}{|Q|}
 \le C\exp\!\left[-c\delta^M
   \frac{(\lambda/K)^{\beta_\delta}}
        {\{\log(e+\lambda/K)\}^{P/2}}\right]
\end{equation}
for all sufficiently large $\lambda$.  Here and below the constants may
also depend on $d,\nu,\kappa,K,A,\delta$, but not on $Q$.

For a nonnegative function, layer-cake integration gives
\[
 \avg{e^{\Lambda v^\rho}}{Q}
 =1+\int_0^\infty
   \Lambda\rho\lambda^{\rho-1}e^{\Lambda\lambda^\rho}
   \frac{|\{v>\lambda\}\cap Q|}{|Q|}\,d\lambda.
\]
The integral over bounded $\lambda$ is finite.  The remaining integral
converges by \eqref{eq:distribution}.  Indeed, choose $\sigma$ with
\[
 \rho<\sigma<\beta_\delta.
\]
For all sufficiently large $\lambda$,
\[
 \frac{(\lambda/K)^{\beta_\delta}}
      {\{\log(e+\lambda/K)\}^{P/2}}
 \ge c(\lambda/K)^\sigma.
\]
The resulting decay $C\exp[-c(\lambda/K)^\sigma]$ dominates
$e^{\Lambda\lambda^\rho}$ for every finite $\Lambda$.  This proves
\eqref{eq:expmoment}.
\end{proof}

\begin{lemma}[Interior heat estimate]\label{lem:interior-heat}
Let $p>(d+2)/2$, and let
$w\in C^{2,1}(Q^-_r(t_0,x_0))$ satisfy
\[
 w_t-\kappa\Delta w=f
 \qquad\text{in }Q^-_r(t_0,x_0),
\]
with $w,f\in L^p(Q^-_r(t_0,x_0))$.  Then
\begin{equation}\label{eq:interior-heat}
 \|w\|_{L^\infty(Q^-_{r/2}(t_0,x_0))}
 \le C(d,p,\kappa)\left[
  \left(\avg{|w|^p}{Q^-_r(t_0,x_0)}\right)^{1/p}
  +r^2\left(\avg{|f|^p}{Q^-_r(t_0,x_0)}\right)^{1/p}
 \right].
\end{equation}
\end{lemma}

\begin{proof}
By parabolic scaling and translation, it suffices to work in
$Q^-_1=(-1,0)\times B_1$.  Choose a smooth cutoff
$\chi(t,x)=\alpha(t)\beta(x)$ such that $\chi=1$ on $Q^-_{1/2}$,
$\alpha=0$ near $t=-1$, $\alpha=1$ on $[-1/2,0]$, $\beta=1$ on
$B_{3/4}$, and $\beta=0$ outside $B_{7/8}$.  For $W=\chi w$,
\[
 (\partial_t-\kappa\Delta)W
 =\chi f+w(\chi_t+\kappa\Delta\chi)
    -2\kappa\operatorname{div}(w\nabla\chi).
\]
Represent $W$ from time $-1$ by the heat kernel.  The spacetime kernel
$\one_{\{s<t\}}G_\kappa(t-s,x-y)$ belongs locally to $L^{p'}$ precisely
when $p>(d+2)/2$.  The terms containing derivatives of $\chi$ are easier:
for $(t,x)\in Q^-_{1/2}$, the support of $\alpha'$ is separated in time
and the supports of $\nabla\beta$ and $\Delta\beta$ are separated in space.
The corresponding heat kernel and gradient heat kernel are therefore
bounded.  H\"older's inequality gives
\[
 \|w\|_{L^\infty(Q^-_{1/2})}
 \le C(d,p,\kappa)
       \bigl(\|w\|_{L^p(Q^-_1)}+\|f\|_{L^p(Q^-_1)}\bigr).
\]
Scaling back to radius $r$ proves \eqref{eq:interior-heat}.
\end{proof}

\begin{corollary}[Local pointwise lifting]\label{cor:lifting}
Assume \eqref{eq:growthassumption} with $0<\rho<\rho_*$.  Fix
$A\ge0$ and $0<R<\infty$.  There exists
\[
 M_R(A)=M(d,\nu,\kappa,K,c_2,Z,\rho,A,R)<\infty
\]
such that, whenever
\[
 0<r\le R,\qquad r^2\le t_0\le T,\qquad
 \avg{v}{Q^-_r(t_0,x_0)}\le A,
\]
one has
\begin{equation}\label{eq:pointbound}
                         v(t_0,x_0)\le M_R(A).
\end{equation}
In particular, $M_R(A)$ is independent of $r$, $t_0$, and $x_0$.
\end{corollary}

\begin{proof}
Let $Q=Q^-_r(t_0,x_0)$ and choose
\[
 p>\frac{d+2}{2},
 \qquad \Lambda=pZ.
\]
Proposition~\ref{prop:expmoment}, the elementary bound
$s^p\le C_{p,\rho,Z}e^{pZs^\rho}$, and $0\le u\le K$ imply
\begin{equation}\label{eq:lpbounds}
 \begin{split}
 \left(\avg{v^p}{Q}\right)^{1/p}&\le C E^{1/p},\\
 \left(\avg{q^p}{Q}\right)^{1/p}
 &\le Kc_2\left(\avg{e^{pZv^\rho}}{Q}\right)^{1/p}
 \le Kc_2E^{1/p},
 \end{split}
\end{equation}
where $E=E(d,\nu,\kappa,K,A,\rho,\Lambda)$.  Applying
Lemma~\ref{lem:interior-heat} to $v_t-\kappa\Delta v=q$ gives
\begin{equation}\label{eq:localestimate}
 \norm{v}{L^\infty(Q^-_{r/2}(t_0,x_0))}
 \le C(d,p,\kappa)\left[
       \left(\avg{v^p}{Q}\right)^{1/p}
       +r^2\left(\avg{q^p}{Q}\right)^{1/p}\right].
\end{equation}
This estimate controls the open cylinder
$Q^-_{r/2}(t_0,x_0)$.  Letting $t\uparrow t_0$ at $x=x_0$ and using the
positive-time continuity of $v$ gives the same bound at $(t_0,x_0)$.
Substitution of \eqref{eq:lpbounds} into
\eqref{eq:localestimate} gives
\[
 v(t_0,x_0)
 \le C(d,p,\kappa)E^{1/p}
       \bigl(C_{p,\rho,Z}+r^2Kc_2\bigr)
 \le C(d,p,\kappa)E^{1/p}
       \bigl(C_{p,\rho,Z}+R^2Kc_2\bigr),
\]
where
\[
 E=E(d,\nu,\kappa,K,A,\rho,pZ).
\]
This proves the asserted uniformity for $0<r\le R$.
\end{proof}

\section{Proof of the main theorem}

The smoothed Martin--Pierre estimate, the fuel-loss lemma, and the discrete
hot--cold expansion used in this section are due to Leontev and Ryzhik
\cite[Lemmas~2.1 and~3.1; Sections~3.2 and~4]{LR26}.  We include the
relevant arguments in the notation and finite-horizon setting used here.
The new input is Corollary~\ref{cor:lifting}, which replaces the
subexponential local estimate \cite[Proposition~2.4]{LR26}.

\begin{lemma}[Smoothed Martin--Pierre estimate]
\label{lem:smoothed-mp}
Let $u,q\ge0$ satisfy $u_t=\nu\Delta u-q$ with $0<\nu\le\kappa$, and
define
\[
 \widetilde H(t_1,t_2)
 :=\int_{t_2}^{t_1}S_\kappa(t_1-s)q(s)\,ds,
 \qquad 0\le t_2\le t_1.
\]
Then, for every $\tau>0$ and $t\ge\tau$,
\begin{equation}\label{eq:smoothed-mp}
 S_\kappa(\tau)\widetilde H(t,t-\tau)
 \le M_{\kappa,\nu}S_\kappa(2\tau)u(t-\tau),
 \qquad
 M_{\kappa,\nu}
 =\left(\frac{2\kappa-\nu}{\kappa}\right)^{d/2}.
\end{equation}
\end{lemma}

The proof below is reproduced from \cite[Lemma~2.1]{LR26}.

\begin{proof}
By the semigroup property, we see
\begin{equation}\label{eq:mp-first}
 S_\kappa(\tau)\widetilde H(t,t-\tau)
 =\int_{t-\tau}^{t}S_\kappa(\tau+t-s)q(s)\,ds.
\end{equation}
For $s\in[t-\tau,t]$, let
\[
 a_s:=\kappa(\tau+t-s)-\nu(t-s),
 \qquad A:=(2\kappa-\nu)\tau.
\]
Since $0<\nu\le\kappa$,
\[
 \kappa\tau\le a_s\le A,
 \qquad
 S_\kappa(\tau+t-s)=S_1(a_s)S_\nu(t-s).
\]
For every nonnegative function $f$, comparison of the Gaussian kernels
at times $a_s$ and $A$ gives
\[
 S_1(a_s)f
 \le\left(\frac{A}{a_s}\right)^{d/2}S_1(A)f
 \le M_{\kappa,\nu}S_1(A)f.
\]
Applying this inequality in \eqref{eq:mp-first} yields
\begin{equation}\label{eq:mp-second}
 S_\kappa(\tau)\widetilde H(t,t-\tau)
 \le M_{\kappa,\nu}S_1(A)
 \int_{t-\tau}^{t}S_\nu(t-s)q(s)\,ds.
\end{equation}
The Duhamel formula for the fuel equation gives
\[
 u(t)=S_\nu(\tau)u(t-\tau)
      -\int_{t-\tau}^{t}S_\nu(t-s)q(s)\,ds.
\]
Since $u\ge0$, the integral in \eqref{eq:mp-second} is bounded above by
$S_\nu(\tau)u(t-\tau)$.  Finally, we have
\[
 S_1(A)S_\nu(\tau)
 =S_1(A+\nu\tau)
 =S_1(2\kappa\tau)
 =S_\kappa(2\tau),
\]
which proves \eqref{eq:smoothed-mp}.
\end{proof}

The following fuel-loss estimate is due to Leontev and Ryzhik
\cite[Lemma~3.1]{LR26}.  We reproduce the argument of
\cite[Section~4]{LR26} to show the dependence of the constants; it
requires only boundedness on compact time intervals.

\begin{lemma}[Fuel loss in a predominantly ignited cylinder]
\label{lem:fuel-loss}
Assume $0<\nu<\kappa$, let $(u,v)$ be a nonnegative solution of
\eqref{eq:system} that is bounded on finite time intervals, and suppose
that $g\in C([0,\infty))$ and $g\ge0$.  Fix $\vartheta>0$ such that
\[
 h_\vartheta:=\inf_{s>\vartheta}g(s)>0.
\]
For every $\eta\in(0,1)$ there exist
\[
 \tau_\eta=\tau_\eta(d,\nu,\kappa,\eta,h_\vartheta)>0,
 \qquad
 \varepsilon_\eta=\varepsilon_\eta(d,\nu,\kappa,\eta)>0
\]
with the following property.  If $\ell_\eta=\sqrt{\tau_\eta}$,
$t\ge\tau_\eta$, and
\begin{equation}\label{eq:fuel-loss-hyp}
 \frac{\bigl|\{v<\vartheta\}\cap
 Q^-_{\ell_\eta}(t,x)\bigr|}
 {|Q^-_{\ell_\eta}(t,x)|}
 \le \varepsilon_\eta,
\end{equation}
then
\begin{equation}\label{eq:fuel-loss-conclusion}
 u(t,x)
 \le\eta\,[S_\kappa(\tau_\eta)u(t-\tau_\eta)](x).
\end{equation}
\end{lemma}

\begin{proof}
Set $a(s,z)=g(v(s,z))$.  On every finite time interval, $a$ is bounded
by the assumed boundedness of $v$ and the continuity of $g$.  Let
$\Gamma(t,x;s,y)$ be the evolution kernel defined by the Feynman--Kac
formula for the bounded measurable nonnegative potential $a$ in
\[
 \partial_t w-\nu\Delta w+a(t,x)w=0.
\]
The Feynman--Kac representation gives the Chapman--Kolmogorov identity,
the representation of $u$ used below, and the domination
\begin{equation}\label{eq:Gamma-free}
 0\le\Gamma(t,x;s,y)\le G_\nu(t-s,x-y).
\end{equation}
We shall choose $\tau=\tau_\eta$ and prove
\begin{equation}\label{eq:Gamma-target}
 \Gamma(t,x;t-\tau,y)
 \le\eta G_\kappa(\tau,x-y)
 \qquad\text{for every }y\in\R^d.
\end{equation}
The conclusion then follows from the representation
\[
 u(t,x)=\int_{\R^d}\Gamma(t,x;t-\tau,y)
 u(t-\tau,y)\,dy.
\]

Choose $L=L_\eta\ge1$ so large that
\begin{equation}\label{eq:L-choice}
 \left(\frac{\kappa}{\nu}\right)^{d/2}
 \exp\!\left(-\frac{\kappa-\nu}{4\kappa\nu}L^2\right)
 \le\eta.
\end{equation}
If $|x-y|\ge L\sqrt\tau$, then \eqref{eq:Gamma-free} and
\[
 \frac{G_\nu(\tau,x-y)}{G_\kappa(\tau,x-y)}
 =\left(\frac{\kappa}{\nu}\right)^{d/2}
 \exp\!\left(-\frac{\kappa-\nu}{4\kappa\nu}
 \frac{|x-y|^2}{\tau}\right)
\]
imply \eqref{eq:Gamma-target}.  It remains to consider
$|x-y|<L\sqrt\tau$. Let $\ell=\sqrt\tau$ and take
\begin{equation}\label{eq:delta-geometric}
 0<\delta<\min\left\{\frac14,\frac{1}{4L}\right\}.
\end{equation}
Duhamel's formula, \eqref{eq:Gamma-free}, and the semigroup property of
$\Gamma$ give
\begin{align}
&G_\nu(\tau,x-y)-\Gamma(t,x;t-\tau,y) \notag\\
&\quad=\int_{t-\tau}^{t}\!\int_{\R^d}
 G_\nu(t-s,x-z)a(s,z)\Gamma(s,z;t-\tau,y)\,dz\,ds \notag\\
&\quad\ge\int_{t-\tau}^{t}\!\int_{\R^d}
 \Gamma(t,x;s,z)a(s,z)\Gamma(s,z;t-\tau,y)\,dz\,ds.
 \label{eq:Gamma-Duhamel}
\end{align}
Since $g$ is continuous, the definition of $h_\vartheta$ also implies
$g(\vartheta)\ge h_\vartheta$.  Restricting the last integral in
\eqref{eq:Gamma-Duhamel} to
\[
 (t-2\delta\tau,t-\delta\tau)\times
 \bigl(B_\ell(x)\cap\{v\ge\vartheta\}\bigr)
\]
therefore yields
\begin{equation}\label{eq:Gamma-loss}
 G_\nu(\tau,x-y)-\Gamma(t,x;t-\tau,y)
 \ge h_\vartheta\delta\tau\,\Gamma(t,x;t-\tau,y)
      -h_\vartheta I_{\mathrm{cold}}
      -h_\vartheta I_{\mathrm{out}},
\end{equation}
where
\begin{align*}
 I_{\mathrm{cold}}
 &=\int_{t-2\delta\tau}^{t-\delta\tau}
   \int_{B_\ell(x)\cap\{v(s,\cdot)<\vartheta\}}
   \Gamma(t,x;s,z)\Gamma(s,z;t-\tau,y)\,dz\,ds,\\
 I_{\mathrm{out}}
 &=\int_{t-2\delta\tau}^{t-\delta\tau}
   \int_{B_\ell(x)^c}
   \Gamma(t,x;s,z)\Gamma(s,z;t-\tau,y)\,dz\,ds.
\end{align*}
Indeed, for every intermediate time $s$,
\[
 \int_{\R^d}\Gamma(t,x;s,z)
 \Gamma(s,z;t-\tau,y)\,dz
 =\Gamma(t,x;t-\tau,y).
\]
For $s\in(t-2\delta\tau,t-\delta\tau)$, set
\[
 \sigma_s=\nu\frac{(t-s)(s-t+\tau)}{\tau},
 \qquad
 m_s=x+\frac{t-s}{\tau}(y-x).
\]
Completing the square gives the Gaussian-bridge identity
\begin{align}
&G_\nu(t-s,x-z)G_\nu(s-t+\tau,z-y) \notag\\
&\qquad=G_\nu(\tau,x-y)(4\pi\sigma_s)^{-d/2}
 \exp\!\left(-\frac{|z-m_s|^2}{4\sigma_s}\right).
 \label{eq:bridge-identity}
\end{align}
Writing $\alpha=(t-s)/\tau\in[\delta,2\delta]$ and using
$\delta<1/4$, we obtain
\begin{equation}\label{eq:sigma-bounds}
 \frac12\nu\delta\tau
 \le\sigma_s=\nu\alpha(1-\alpha)\tau
 \le2\nu\delta\tau.
\end{equation}
It follows from \eqref{eq:Gamma-free},
\eqref{eq:bridge-identity}, and \eqref{eq:sigma-bounds} that
\begin{align}
 I_{\mathrm{cold}}
 &\le C(\delta\tau)^{-d/2}G_\nu(\tau,x-y)
 \bigl|\{v<\vartheta\}\cap Q^-_\ell(t,x)\bigr| \notag\\
 &\le C\varepsilon\tau\delta^{-d/2}G_\nu(\tau,x-y),
 \label{eq:Icold-bound}
\end{align}
provided the fraction in \eqref{eq:fuel-loss-hyp} is at most
$\varepsilon$.  Here and below $C$ depends only on $d$ and $\nu$; we
used $|Q^-_\ell(t,x)|=\tau|B_1|\ell^d$ and $\ell^2=\tau$. Moreover, $|x-y|<L\ell$ and \eqref{eq:delta-geometric} imply
$|m_s-x|\le2\delta L\ell\le\ell/2$.  Hence
$z\in B_\ell(x)^c$ implies $|z-m_s|\ge\ell/2$.  The bounds
\eqref{eq:sigma-bounds} consequently give
\begin{align}
&\int_{B_\ell(x)^c}(4\pi\sigma_s)^{-d/2}
 \exp\!\left(-\frac{|z-m_s|^2}{4\sigma_s}\right)dz \notag\\
&\qquad\le C(\delta\ell^2)^{-d/2}
 \int_{|w|\ge\ell/2}
 \exp\!\left(-\frac{|w|^2}{8\nu\delta\ell^2}\right)dw
 \le Ce^{-c/\delta},
 \label{eq:bridge-tail}
\end{align}
where $C,c>0$ depend only on $d$ and $\nu$.  Integrating over the time
interval of length $\delta\tau$ yields
\begin{equation}\label{eq:Iout-bound}
 I_{\mathrm{out}}
 \le C\delta\tau e^{-c/\delta}G_\nu(\tau,x-y).
\end{equation}
Combining \eqref{eq:Gamma-loss}, \eqref{eq:Icold-bound}, and
\eqref{eq:Iout-bound}, and using
$G_\nu(\tau,\cdot)\le(\kappa/\nu)^{d/2}G_\kappa(\tau,\cdot)$, gives
\begin{equation}\label{eq:Gamma-final-estimate}
 \Gamma(t,x;t-\tau,y)
 \le C_0\left(
 \frac{1}{h_\vartheta\delta\tau}
 +\varepsilon\delta^{-d/2-1}
 +e^{-c_0/\delta}\right)G_\kappa(\tau,x-y),
\end{equation}
where $C_0,c_0>0$ depend only on $d,\nu,\kappa$. We now fix the parameters.  After choosing $L=L_\eta$ by
\eqref{eq:L-choice}, choose $\delta=\delta_\eta$ satisfying
\eqref{eq:delta-geometric} and
\[
 C_0e^{-c_0/\delta_\eta}\le\frac\eta3.
\]
Set
\[
 \varepsilon_\eta
 =\min\left\{1,
 \frac{\eta}{3C_0}\delta_\eta^{d/2+1}\right\},
 \qquad
 \tau_\eta
 =\max\left\{1,
 \frac{3C_0}{\eta h_\vartheta\delta_\eta}\right\}.
\]
Then \eqref{eq:Gamma-final-estimate} proves
\eqref{eq:Gamma-target} when
$|x-y|<L_\eta\sqrt{\tau_\eta}$, while \eqref{eq:L-choice} proves it in
the complementary region.  Integrating \eqref{eq:Gamma-target} against
$u(t-\tau_\eta,y)\ge0$ proves \eqref{eq:fuel-loss-conclusion}.
\end{proof}

\begin{proof}[Proof of Theorem~\ref{thm:main}]
Let $[0,T_{\max})$ be the maximal lifespan supplied by
Lemma~\ref{lem:local}, and fix an arbitrary $T<T_{\max}$.  We use the
finite-horizon extensions introduced at the beginning of Section~2.  Thus
$u,v,q,F$, and $H$ below denote the extended quantities, agree with the
original solution on $[0,T]$, and satisfy $0\le u,F\le K$ on all positive
times.  Every constant below is independent of $T$.

Choose $\vartheta>0$ such that
\[
                 h_\vartheta:=\inf_{s>\vartheta}g(s)>0,
\]
and fix $\eta=1/2$.  The fuel-loss estimate supplies
$\tau=\tau_\eta>0$, $\ell=\sqrt\tau$, and
$\varepsilon_\eta>0$.  Let $C_{\mathrm B}$ be the constant in
\eqref{eq:globalBMO}.  The parabolic John--Nirenberg inequality
\cite{JN61,Aim88}
gives, whenever $\avg{v}{Q}>A>\vartheta$,
\[
 \frac{|\{v<\vartheta\}\cap Q|}{|Q|}
 \le C_{\mathrm{JN}}
    \exp\!\left(-\frac{c_{\mathrm{JN}}(A-\vartheta)}{C_{\mathrm B}}\right).
\]
Choose $A=A_\eta$ so that the last expression is at most
$\varepsilon_\eta$.  For $\tau\le s\le T$, define
\[
 U_s=\left\{x:\avg{v}{Q^-_\ell(s,x)}>A_\eta\right\},
\]
and set $U_s=\varnothing$ otherwise.  The averaging map is continuous in
$x$, so $U_s$ is open.  Lemma~\ref{lem:fuel-loss} yields
\begin{equation}\label{eq:hotfuel}
 u(s,x)\le\eta\,[S_\kappa(\tau)u(s-\tau,\cdot)](x),
 \qquad \tau\le s\le T,\quad x\in U_s.
\end{equation}
On the complement of $U_s$, Corollary~\ref{cor:lifting}, with
$M_\eta:=M_\ell(A_\eta)$, gives
\begin{equation}\label{eq:coldpoint}
                         v(s,x)\le M_\eta,
 \qquad \tau\le s\le T,\quad x\notin U_s.
\end{equation}
For $0\le s_2\le s_1$, define
\[
 \widetilde H(s_1,s_2)
 :=\int_{s_2}^{s_1}S_\kappa(s_1-s)q(s,\cdot)\,ds,
 \qquad H(s)=\widetilde H(s,0),
\]
and, for $s\ge\tau$, set
\[
                         V(s):=S_\kappa(\tau)H(s-\tau).
\]
Then $0\le V(s)\le H(s)\le v(s)$, and for $s\ge2\tau$,
\begin{equation}\label{eq:Vduhamel}
 V(s)=S_\kappa(\tau)V(s-\tau)
      +S_\kappa(\tau)\widetilde H(s-\tau,s-2\tau).
\end{equation}
The smoothed Martin--Pierre estimate is valid for every positive time
step.  With the present value of $\tau$, it states that
\begin{equation}\label{eq:MPsource}
 S_\kappa(\tau)\widetilde H(s,s-\tau)
 \le M_{\kappa,\nu}S_\kappa(2\tau)u(s-\tau),
 \qquad
 M_{\kappa,\nu}=\left(\frac{2\kappa-\nu}{\kappa}\right)^{d/2},
 \qquad s\ge\tau.
\end{equation}

Fix $3\tau<t\le T+\tau$ and set
\[
 \mathsf N=\left\lfloor\frac{t}{\tau}\right\rfloor-2,
 \qquad t_0=t-\mathsf N\tau.
\]
Then $\mathsf N\ge1$ and
\begin{equation}\label{eq:t0range}
                         2\tau\le t_0<3\tau.
\end{equation}
In particular, the earliest time used below is
$t-(\mathsf N+2)\tau=t_0-2\tau\ge0$.
Let $\one_E$ also denote multiplication by the indicator of $E$.  Define
\[
 \mathsf T_0^t=I,\qquad \mathsf F_0^t=0,
\]
and, for $0\le n\le \mathsf N+1$,
\[
 \mathsf T_{n+1}^t
 =\mathsf T_n^t\circ S_\kappa(\tau)
       \circ\one_{U_{t-(n+1)\tau}},
 \qquad
 \mathsf F_{n+1}^t
 =\mathsf T_n^t\circ S_\kappa(\tau)
       \circ\one_{U_{t-(n+1)\tau}^c}.
\]
These operators preserve positivity, have $L^\infty$ operator norm at most
one, and satisfy
\begin{equation}\label{eq:TFidentities}
 \mathsf T_n^tS_\kappa(\tau)
   =\mathsf T_{n+1}^t+\mathsf F_{n+1}^t,
 \qquad
 \mathsf F_{n+1}^t\one
   =\mathsf T_n^t\one-\mathsf T_{n+1}^t\one.
\end{equation}

Iterating \eqref{eq:Vduhamel} $\mathsf N$ times yields
\begin{align}
 V(t)
 &=\mathsf T_{\mathsf N}^tV(t_0)
   +\sum_{n=1}^{\mathsf N}\mathsf F_n^tV(t-n\tau) \notag\\
 &\quad
   +\sum_{n=0}^{\mathsf N-1}\mathsf T_n^tS_\kappa(\tau)
       \widetilde H\bigl(t-(n+1)\tau,t-(n+2)\tau\bigr).
 \label{eq:iteratedV}
\end{align}
The first term has a bound depending only on the data.  Indeed, write
$a=t_0-2\tau$ and $r=t_0-\tau$, so that $0\le a<\tau$ and
$r-a=\tau$.  The semigroup identity gives
\begin{align*}
 V(t_0)
 &=S_\kappa(2\tau)H(a)
   +S_\kappa(\tau)\widetilde H(r,a).
\end{align*}
The second term is at most $M_{\kappa,\nu}K$ by
\eqref{eq:MPsource}.  If $a>0$, the same estimate with time step $a$
shows that
\[
 S_\kappa(2\tau)H(a)
 =S_\kappa(2\tau-a)\bigl[S_\kappa(a)H(a)\bigr]
 \le M_{\kappa,\nu}K;
\]
for $a=0$ this term vanishes.  Thus
\begin{equation}\label{eq:initialV}
 \|\mathsf T_{\mathsf N}^tV(t_0)\|_{L^\infty}
 \le2M_{\kappa,\nu}K.
\end{equation}
By \eqref{eq:coldpoint}, $V\le M_\eta$ on the cold set.  Hence
\[
 \sum_{n=1}^{\mathsf N}\mathsf F_n^tV(t-n\tau)
 \le M_\eta\sum_{n=1}^{\mathsf N}\mathsf F_n^t\one
 =M_\eta\bigl(\one-\mathsf T_{\mathsf N}^t\one\bigr)
 \le M_\eta.
\]

Denote the last sum in \eqref{eq:iteratedV} by $I_3$.  Equations
\eqref{eq:MPsource} and \eqref{eq:TFidentities} give
\[
 I_3\le M_{\kappa,\nu}
       (\mathfrak B_1+\mathfrak B_2+\mathfrak B_3),
\]
where
\begin{align*}
 \mathfrak B_1
 &=\sum_{n=0}^{\mathsf N-1}
     \mathsf T_{n+2}^tu(t-(n+2)\tau)
   =\sum_{j=2}^{\mathsf N+1}\mathsf T_j^tu(t-j\tau),\\
 \mathfrak B_2
 &=\sum_{n=0}^{\mathsf N-1}
     \mathsf F_{n+2}^tu(t-(n+2)\tau),\\
 \mathfrak B_3
 &=\sum_{n=0}^{\mathsf N-1}
     \mathsf F_{n+1}^tS_\kappa(\tau)u(t-(n+2)\tau).
\end{align*}
Since $0\le u\le K$ and $S_\kappa(\tau)u\le K$,
\begin{equation}\label{eq:B23}
 \mathfrak B_2+\mathfrak B_3
 \le K\sum_{j=2}^{\mathsf N+1}\mathsf F_j^t\one
     +K\sum_{j=1}^{\mathsf N}\mathsf F_j^t\one
 \le2K.
\end{equation}

For $2\le j\le \mathsf N+1$, \eqref{eq:t0range} gives
\[
 t-j\tau\ge t_0-\tau\ge\tau,
 \qquad
 t-j\tau\le T,
 \qquad
 t-(j+1)\tau\ge t_0-2\tau\ge0.
\]
Applying \eqref{eq:hotfuel} inside the definition of
$\mathsf T_j^t$ and then using \eqref{eq:TFidentities}, we obtain
\begin{align*}
 \mathfrak B_1
 &\le\eta\sum_{j=2}^{\mathsf N+1}
      \mathsf T_j^tS_\kappa(\tau)u(t-(j+1)\tau)\\
 &=\eta\sum_{j=3}^{\mathsf N+2}
      \bigl(\mathsf T_j^t+\mathsf F_j^t\bigr)u(t-j\tau)\\
 &\le\eta\mathfrak B_1+\eta K.
\end{align*}
Indeed, $t-(\mathsf N+2)\tau=t_0-2\tau<\tau$, so
$U_{t-(\mathsf N+2)\tau}=\varnothing$ and
$\mathsf T_{\mathsf N+2}^t=0$.  The shifted $\mathsf T$-sum is therefore
bounded by $\mathfrak B_1$, while the $\mathsf F$-sum is at most $K$ by
\eqref{eq:TFidentities}.  Consequently,
\begin{equation}\label{eq:B1}
                 \mathfrak B_1\le\frac{\eta K}{1-\eta}.
\end{equation}
Combining \eqref{eq:iteratedV}, \eqref{eq:B23}, and \eqref{eq:B1} gives
\[
 \|V(t)\|_{L^\infty}
 \le 2M_{\kappa,\nu}K+M_\eta
   +M_{\kappa,\nu}\left(2K+\frac{\eta K}{1-\eta}\right),
 \qquad 3\tau<t\le T+\tau.
\]
For $\tau<s\le2\tau$, set $r=s-\tau\in(0,\tau]$.  The same estimate with
time step $r$ gives
\[
 V(s)=S_\kappa(\tau-r)\bigl[S_\kappa(r)H(r)\bigr]
 \le M_{\kappa,\nu}K,
\]
and $V(\tau)=0$.  The computation leading to \eqref{eq:initialV} applies
with $t_0=s$ for every $s\in[2\tau,3\tau]$ and gives
\[
 \|V(s)\|_{L^\infty}\le2M_{\kappa,\nu}K
\]
there.  Hence $V$ is uniformly bounded on $[\tau,T+\tau]$.  Since
\[
 S_\kappa(\tau)v(s)
 =S_\kappa(s+\tau)v_0+V(s+\tau)
 \le K+V(s+\tau),
\]
we have
\begin{equation}\label{eq:smoothed}
 \sup_{0\le s\le T,\,x\in\R^d}
       [S_\kappa(\tau)v(s,\cdot)](x)<\infty.
\end{equation}

Set
\[
 c_\tau:=\inf_{|z|\le1}G_\kappa(\tau,z)>0.
\]
It follows from \eqref{eq:smoothed} that
\[
 \sup_{0\le s\le T,\,x\in\R^d}\int_{B_1(x)}v(s,y)\,dy
 \le c_\tau^{-1}
 \sup_{0\le s\le T,\,x\in\R^d}
       [S_\kappa(\tau)v(s,\cdot)](x)<\infty.
\]
Integration over $s\in(t-1,t)$ gives
\begin{equation}\label{eq:uniformmeans}
 \sup_{1\le t\le T,\,x\in\R^d}\avg{v}{Q^-_1(t,x)}<\infty.
\end{equation}
Corollary~\ref{cor:lifting} applied to \eqref{eq:uniformmeans} bounds
$v(t,x)$ uniformly for $1\le t\le T$.

The same structural constants also control the initial time interval.
Fix $0<t_0\le\min\{1,T\}$ and $x_0\in\R^d$, and set
\[
                 r^2=\frac{t_0}{2},
 \qquad Q=Q^-_r(t_0,x_0)
          =(r^2,2r^2)\times B_r(x_0).
\]
For $s\in[r^2,2r^2]$, let $a=s-r^2\in[0,r^2]$.  The Duhamel identity
gives
\[
 S_\kappa(r^2)H(s)
 =S_\kappa(2r^2)H(a)
   +S_\kappa(r^2)\widetilde H(s,a).
\]
The smoothed Martin--Pierre estimate, first with time step $r^2$ and then
with time step $a$, yields
\[
 S_\kappa(r^2)\widetilde H(s,a)
 \le M_{\kappa,\nu}S_\kappa(2r^2)u(a)
 \le M_{\kappa,\nu}K
\]
and, if $a>0$,
\[
 S_\kappa(2r^2)H(a)
 =S_\kappa(2r^2-a)\bigl[S_\kappa(a)H(a)\bigr]
 \le M_{\kappa,\nu}K.
\]
The latter term vanishes when $a=0$.  Consequently,
\begin{equation}\label{eq:shorttimesmoothed}
 [S_\kappa(r^2)v(s,\cdot)](x_0)
 \le (1+2M_{\kappa,\nu})K.
\end{equation}
For $y\in B_r(x_0)$,
\[
 G_\kappa(r^2,x_0-y)
 \ge (4\pi\kappa)^{-d/2}r^{-d}
       \exp\!\left(-\frac1{4\kappa}\right).
\]
It follows from \eqref{eq:shorttimesmoothed} that
\[
 \avg{v}{Q}
 \le A_0
 :=\frac{(4\pi\kappa)^{d/2}e^{1/(4\kappa)}}{|B_1|}
       (1+2M_{\kappa,\nu})K.
\]
Since $0<r\le1/\sqrt2$, Corollary~\ref{cor:lifting} with $R=1$ gives
\[
                         v(t_0,x_0)\le M_1(A_0),
\]
uniformly for $0<t_0\le\min\{1,T\}$ and $x_0\in\R^d$.  Together with
$v(0,\cdot)=v_0\le K$, this proves
\[
 \sup_{0\le t\le T}\|v(t)\|_{L^\infty}\le C,
\]
where $C$ is independent of $T<T_{\max}$.  Since $T$ was arbitrary, the
same bound holds on the maximal lifespan.  The continuation alternative
\eqref{eq:continuation} therefore gives $T_{\max}=\infty$ and completes the
proof.
\end{proof}

\section{Acknowledgment}
The author is grateful to Lenya Ryzhik for useful discussions. The author acknowledges use of AI for literature search, proof refinement, and copyediting. This work was partially supported by the Simons Foundation (No. 0007730).

\end{document}